\documentclass[a4paper,10pt]{amsart}
\usepackage[utf8x]{inputenc}

\usepackage{amssymb}
 \usepackage{amsmath,amsfonts,amsthm,amstext}
\usepackage[all,cmtip]{xy}
\usepackage[dvips]{graphicx}
\usepackage{graphics}
\usepackage{calrsfs}
\usepackage{hyperref}
\usepackage{xcolor}
 \usepackage{breqn}

\newtheorem{theorem}{Theorem}[section]
\newtheorem{lemma}[theorem]{Lemma}
\newtheorem{prop}[theorem]{Proposition}

\newtheorem{crlr}[theorem]{Corollary}

\theoremstyle{definition}

\newtheorem{rem}[theorem]{Remark}

\newcommand{\mf}[1]{\mathfrak{#1}}
\newcommand{\mc}[1]{\mathcal{#1}}
\newcommand{\mb}[1]{\mathbb{#1}}

\title{Weak commutativity and nilpotency}
 \author{Luis Augusto de Mendon\c ca}
 \address{Department of Mathematics, University of Campinas (UNICAMP), rua S\'{e}rgio Buarque de Holanda, 651, 13083-859, Campinas-SP, Brazil.}
 \email{demendoncaluisaugusto@gmail.com}
\begin{document}

 \begin{abstract}
  We continue the analysis of the weak commutativity construction for Lie algebras. This is the Lie algebra $\chi(\mathfrak{g})$ generated by two isomorphic copies $\mathfrak{g}$ and $\mathfrak{g}^{\psi}$ of a fixed Lie algebra, subject to the relations $[x,x^{\psi}]=0$ for all $x \in \mathfrak{g}$. In this article we study the ideal $L =L(\mf{g})$ generated by $x-x^{\psi}$ for all $x \in\mf{g}$. We obtain an (infinite) presentation for $L$ as a Lie algebra, and we show that in general it cannot be reduced to a finite one. With this in hand, we study the question of nilpotency. We show that if $\mf{g}$ is nilpotent of class $c$, then $\chi(\mf{g})$ is nilpotent of class at most $c+2$, and this bound can improved to $c+1$ if $\mf{g}$ is $2$-generated or if $c$ is odd. We also obtain concrete descriptions of $L(\mf{g})$ (and thus of $\chi(\mf{g})$) if $\mf{g}$ is free nilpotent of class $2$ or $3$. Finally, using methods of Gröbner-Shirshov bases we show that the abelian ideal $R(\mf{g}) = [\mf{g}, [L, \mf{g}^{\psi}]]$ is infinite-dimensional if $\mf{g}$ is free of rank at least $3$. 
 \end{abstract}


\maketitle

\begin{section}{Introduction}
Given a Lie algebra $\mf{g}$ over a field $K$, with $char(K)\neq2$, the weak commutativity construction 
is defined as
\[ \chi(\mf{g}) = \langle \mf{g}, \mf{g}^{\psi} |\ \ [x,x^{\psi}]= 0 \hbox{ for all } x \in \mf{g} \rangle,\]
where $\mf{g}^{\psi}$ is an isomorphic copy of $\mf{g}$ and the isomorphism is written as $x \mapsto x^{\psi}$.
This is the Lie algebra version of a group-theoretic construction defined by Sidki \cite{Sid} and studied subsequently 
by many authors \cite{BLN}, \cite{BriKoc},  \cite{GupRocSid}, \cite{KocSid}, \cite{LimOli},  \cite{Roc}.

The study of the Lie algebra construction was started by the author in \cite{Men}. There it is shown that most of the group-theoretic results from the articles above admit an analogue for Lie algebras, though sometimes the proofs are completely different. For instance, it is shown that
if $\mf{g}$ has one of the following properties, then so does $\chi(\mf{g})$:
\begin{enumerate}
 \item $\mf{g}$ is finite-dimensional;
 \item $\mf{g}$ is finitely presentable;
 \item $\mf{g}$ is solvable and of homological type $FP_{\infty}$.
\end{enumerate}
On the other hand, $\chi(-)$ does not preserve in general the $FP_m$ properties, as we can see by considering free Lie algebras of finite rank.

Most of the proofs of these results involve the analysis of some special ideals of $\chi(\mf{g})$. We can define them by nice homomorphisms. Let $\alpha: \chi(\mf{g}) \to \mf{g}$ be defined by
\[ \alpha(x) = \alpha(x^{\psi}) = x,\]
for all $x \in \mf{g}$. Similarly, define $\beta: \chi(\mf{g}) \to \mf{g} \oplus \mf{g}$ by
\[\beta(x)= (x,0), \  \ \beta(x^{\psi}) = (0,x)\]
for all $x \in \mf{g}$. Finally, $\rho: \chi(\mf{g}) \to \mf{g} \oplus \mf{g} \oplus \mf{g}$ is defined by:
\[\rho(x)= (x,x,0), \  \ \rho(x^{\psi}) = (0,x,x)\]
for all $x \in \mf{g}$. We then define:
\[ L(\mf{g}) := ker(\alpha), \  \ D(\mf{g}) := ker(\beta), \  \ W(\mf{g}) := ker(\rho).\]
When there is no risk of ambiguity we write simply $L$, $D$ and $W$.

One can show that $[L,D]=0$ for any Lie algebra $\mf{g}$. It follows then that $W$ is an \textit{abelian ideal} of $\chi(\mf{g})$, since clearly $W = L \cap D$. It is not completely clear though when $W$ is finite-dimensional. Theorem 1.3 in 
\cite{Men} shows that this is the case if $\mf{g}$ is finitely presented (or, more generally, of homological type $FP_2$) and $\mf{g}'/\mf{g}''$ is finite-dimensional, but the converse is not true: if $\mf{g}$ is free of rank $2$, then $W=0$ (by Theorem 1.6 and Proposition 7.3 in \cite{Men}), but $\mf{g'}/\mf{g''}$ is of infinite dimension. 

In this article we focus on $L(\mf{g})$. We start by describing a finite generating set for it as a Lie algebra, for any finitely generated Lie algebra $\mf{g}$. We write then a presentation for $L(\mf{g})$ in terms of these generators. This is Theorem \ref{Thm1} and the remarks following it. We do not state it here completely because that would require introducing a lot of notation. The set of generators is
\[\{x-x^{\psi}, [x,y]-[x,y]^{\psi}\}_{x,y},\]
where $x$ and $y$ run through a generating set for $\mf{g}$, and 
 defining relations come from some manipulation of the identity
\begin{equation}  \label{fundiden}
[x-x^{\psi}, [y-y^{\psi}, z-z^{\psi}]] = [x,[y,z]] - [x^{\psi},[y^{\psi},z^{\psi}]],
\end{equation}
which holds for all $x,y,z \in \mf{g}$. 

Even if $\mf{g}$ is finitely presented, our presentation will be in general infinite. This is expected: $L(\mf{g})$ is not finitely presented if $\mf{g}$ is free of rank at least $2$ (see Proposition \ref{H2Linfinite}).

What makes this presentation interesting is that it helps when performing computations on $L(\mf{g})$, and since 
$\chi(\mf{g}) \simeq L \rtimes \mf{g}$, this allows us to obtain results on the structure of $\chi(\mf{g})$. We used this to study the question of nilpotency.

\begin{theorem} \label{Thm2}
Suppose that $\mf{g}$ is nilpotent of class $c$. 
\begin{enumerate}
 \item If $c$ is odd, then $\chi(\mf{g})$ is nilpotent of class at most $c+1$.
 \item If $c$ is even, then $\chi(\mf{g})$ is nilpotent of class at most $c+2$.
 \end{enumerate}
\end{theorem}
This should be compared with the article of Gupta, Rocco and Sidki \cite{GupRocSid}, where the nilpotency of the group-theoretic construction is studied. There it is shown that if a group $G$ is nilpotent of class $c$, then $\chi(G)$ is nilpotent of class at most $max\{ c+2, d(G)\}$, where $d(G)$ is the minimal number of generators of $G$.

Another consequence is that we were able to obtain a concrete description of $\chi(\mf{n}_{m,c})$, where $\mf{n}_{m,c}$ is a free nilpotent Lie algebra of class $c$ and rank $m$, for $c =2$ or $3$. 

\begin{crlr}  \label{Crlr1}
If $\mf{h} =  \mf{n}_{m,2}$, then $L(\mf{h})$ is free nilpotent of rank $m+ \binom{m}{2}$ and class $2$. In particular, we have
\[dim \chi(\mf{h}) = 2k+ \binom{k}{2},\]
where $k= m + \binom{m}{2}$, and $\chi(\mf{h})$ is nilpotent of class exactly $4$ if $m \geq 3$.
\end{crlr}


\begin{crlr}  \label{Crlr2}
If $\mf{h} =  \mf{n}_{m,3}$, then $L(\mf{h})$ is a central extension of $K^{m \binom{m}{2}}$ by 
$\mf{n}_{m,4} \oplus \mf{n}_{\binom{m}{2}, 2}$. In particular, $\chi(\mf{h})$ is nilpotent of class exactly $4$.
\end{crlr}

The dimension of $\chi(\mf{n}_{m,3})$ can be computed directly by the proposition above together with Witt's dimension formula.

The weak commutativity construction is intimately connected with non-abelian tensor products or, more specifically, with non-abelian exterior squares of Lie algebras (as defined in \cite{Ell1} and \cite{Ell2}). In fact, there is a quotient of $W(\mf{g})$ that is isomorphic to a certain distinguished ideal of $\mf{g} \wedge \mf{g}$, which in turn is isomorphic to the Schur multiplier $H_2(\mf{g};K)$. The ideal of $\chi(\mf{g})$ realizing this quotient can be defined as follows:
\[ R = R(\mf{g}) := [\mf{g}, [L,\mf{g}^{\psi}]] \subseteq W(\mf{g}).\]
The isomorphism between $W/R$ and the Schur multiplier of $\mf{g}$ is Theorem 1.6 in \cite{Men}.

As it turns out, $R(\mf{g})=0$ very often. Sufficient conditions for that are: $\mf{g}$ is abelian, $\mf{g}$ is perfect, or $\mf{g}$ is generated by two elements. The unique examples that we had where $R(\mf{g}) \neq 0$ were obtained with the aid of GAP \cite{GAP}. 

From Corollary \ref{Crlr1} we obtain:
\[ dim R(\mf{n}_{m,2}) = \frac{1}{24}(3m^4-2m^3-15m^2+14m),\]
which grows with $m$ and is non-zero for $m \geq 3$. Interestingly enough, we have
\[ dim R(\mf{n}_{m,2}) = dim R(\mf{n}_{m,3})\]
for all $m$. 


We do not have versions of Corollaries \ref{Crlr1} and \ref{Crlr2} for higher nilpotency classes. The reason for this is that some defining relators of $L(\mf{g})$ (those written in \eqref{Lrel1} in the text) are automatically trivial when $\mf{g}$ is nilpotent of class at most $3$, but not otherwise. Thus we do not think that we can write such a concrete description of $L(\mf{n}_{m,4})$, for instance.
%
%

Finally, using the concept of Gröbner-Shirshov bases applied to the presentation of $L(\mf{g})$, we deduce the following.
\begin{theorem} \label{Rinf}
 If $\mf{g}$ is free non-abelian of rank at least $3$, then $R(\mf{g})$ is infinite-dimensional.
\end{theorem}
We obtain from the theorem the first example of a finitely presentable Lie algebra with $W(\mf{g})$ of infinite dimension. A proof that the analogue of Theorem \ref{Rinf} also holds for groups has been announced by Bridson and Kochloukova (see \cite[Question~5.3]{BriKoc}).

Gröbner-Shirshov bases are designed in a way that we can in theory find a basis of a finitely presented Lie algebra, but in our case we only deduce that a carefully chosen infinite subset of $R(\mf{g})$ is linearly independent. 

The following corollary is immediate.
\begin{crlr}
If $\mf{g}$ is free non-abelian of rank at least $3$, then $\chi(\mf{g})$ is of infinite cohomological dimension.
\end{crlr}
\end{section}

\begin{section}{Preliminaries}
Let $\mb{N}$ denote the set $\{1,2,3, \ldots \}$ of positive integers. We fix a field $K$, with $char(K) \neq 2$, and consider only Lie algebras over $K$.
 
 We will use the convention of right-normed brackets, that is:
 \[ [x_1, x_2, \dots, x_{n-1}, x_n] = [x_1, [x_2, \dots, [x_{n-1},x_n] \ldots]]\]
 for any elements $x_1, \ldots, x_n$ of a Lie algebra.
 
 For any subset $X$ of a Lie algebra $\mf{g}$, we denote by $\langle X \rangle$ (resp. $\langle\langle X\rangle \rangle$) the subalgebra (resp. the ideal) of $\mf{g}$ generated by $X$. The universal enveloping algebra of $\mf{g}$ is denoted by $\mc{U}(\mf{g})$.
 
 Let $\mf{g}$ be a free Lie algebra of rank $n$. Denote by $\gamma_k(\mf{g})$ the $k$-th term of the lower central series, that is, $\gamma_1(\mf{g})= \mf{g}$ and $\gamma_{k+1}(\mf{g}) = [\mf{g}, \gamma_k(\mf{g})]$ for all $k$. \textit{Witt's dimension formula} (\cite{Bou}, Théorème 3 of II.3.3)  tell us the dimension of the successive quotients of this series:
  \[dim (\gamma_k(\mf{g})/\gamma_{k+1}(\mf{g})) = \frac{1}{k} (\sum_{d|k} \mu(d) n^{k/d}),\]
 where $\mu$ is the Möbius function. Notice that we can use this formula to compute the dimension of the free nilpotent Lie algebra of class $c$ on $n$ generators.
 
 \begin{subsection}{Homology of Lie algebras}
The homology of a Lie algebra is defined in terms of its universal enveloping algebra. For a $\mf{g}$-module $A$, we put
\[ H_i(\mf{g};A) = Tor_i^{\mc{U}(\mf{g})}(A,K).\]
We notice that if $\mf{g}$ is a finitely presented Lie algebra, then the trivial $\mc{U}(\mf{g})$-module $K$ admits a free resolution that is finitely generated up to degree $2$. From this it follows that the homologies $H_i(\mf{g};K)$ are finite-dimensional for $i \leq 2$. 

For a short exact sequence 
$\mf{h} \rightarrowtail \mf{g} \twoheadrightarrow \mf{q}$ of Lie algebras and a $\mc{U}(\mf{g})$-module $A$, there is a Lyndon-Hochschild-Serre (LHS) spectral sequence
\[ E_{p,q}^2= H_p(\mf{q}; H_q(\mf{h}; A)) \Rightarrow H_{p+q}(\mf{g}; A),\]
which is convergent and concentrated in the first quadrant. As usual, the differential $d^r$ of the $r$-th page has bidegree $(-r,r-1)$. If $A=K$ is the trivial $\mc{U}(\mf{g})$-module and $\mf{h}$ is a \textit{central} ideal of $\mf{g}$, then the associated $5$-term exact sequence can be written as
\[H_2(\mf{g}; K) \to H_2(\mf{q}; K) \to \mf{h} \to H_1(\mf{g}; K) \to H_1(\mf{q};K) \to 0.\]
For details see \cite[Chapter~7]{Wei}, for instance.
%
\end{subsection}

\begin{subsection}{Gröbner-Shirshov bases} \label{Shi}
We recall here briefly the theory of Gröbner-Shirshov bases, following the exposition in \cite{BKbook}. The original arguments are due to Shirshov \cite{Shirshov2}, and the modern approach was initiated by Bokut \cite{Bokut}.

Let $\mf{g}$ be the free Lie algebra with free basis $X = \{x_1, \ldots, x_n\}$. Associative words with letters in $X$ are ordered lexicographically with $x_1 > \cdots > x_n$ and $u > v$ if $u$ is a initial subword of $v$. One of such words $w$ is \textit{regular} (or a \textit{Lyndon-Shirshov associative word}) if $w=uv$ implies $uv>_{lex} vu$, for  non-trivial subwords $u,v$. A non-associative word $[w]$ is \textit{regular} (or a \textit{Lyndon-Shirshov non-associative word}) if the associative word $w$ obtained by removing all brackets is associative regular and:
\begin{enumerate}
 \item If $[w]=[u][v]$, then both $[u]$ and $[v]$ are non-associative regular, and
 \item If $[w]=[ [u_1][u_2] ][v]$, then $u_2 \leq_{lex} v$. 
\end{enumerate}
For any regular associative word $w$ there is a unique bracketing $(w)$ that is non-associative regular, and the set of all non-associative regular words is a basis of the free Lie algebra \cite{Shirshov1}. For short we will also call these words \textit{monomials}.

Let $d: X \to \mb{N}$ be any function. Consider its extension to the set of regular associative words:
\[ d( x_{i_1} \cdots x_{i_m} )  = d( x_{i_1} ) +  \ldots + d(x_{i_m} ).\] 
We say that $d(w)$ is the \textit{degree} of $w$.
We consider the \textit{weighted deg-lex} ordering on the set of associative regular words: $w_1 \preceq w_2$ if
$d(w_1) < d(w_2)$, or if  $d(w_1)= d(w_2)$ but $w_1 \leq_{lex} w_2$. The degree function and the associated ordering can be considered  also  for non-associative regular words via the bijection given by the unique bracketing.

Any $f \in \mf{g}$ may be written as a linear combination of regular non-associative words. We denote by $\bar{f}$ the highest (with respect to the weighted deg-lex ordering) corresponding regular associative word appearing with non-zero coefficient. We say that $\bar{f}$ is the \textit{associative carrier} of $f$. If the coefficient of $\bar{f}$ in $f$ is $1 \in K$, we say that $f$ is \textit{monic}.

Suppose that $f,g \in \mf{g}$ are monic and satisfy $\bar{f}= ab$ and $\bar{g}=bc$ for some associative words $a,b,c$. Let $u=abc$. Notice that $u$ is regular. As in \cite[Lemma~2.11.15]{BKbook}, we can consider two special bracketings $[u]_1$ and $[u]_2$ for the word $u$ such that $[u]_1$ extends the regular bracketing $(\bar{f})$ of $\bar{f}$ and $[u]_2$ extends the regular bracketing $(\bar{g})$ of $\bar{g}$, and $\bar{ [u]}_1  = \bar{[u]}_2=u$.

Let $u_1$ and $u_2$ be the elements of $\mf{g}$ obtained by from $[u]_1$ and $[u]_2$ by substituting $(\bar{f})$ and $(\bar{g})$ with $f$ and $g$, respectively. The \textit{first-order composition} $(f,g)^I_u$ of $f$ and $g$ with respect to $u$ defined as
\[(f,g)^I_u = u_1 - u_2.\]

Similarly, suppose that $\bar{g}$ is a subword of $\bar{f}$ for some monic elements $f,g \in \mf{g}$. Again by \cite[Lemma~2.11.15]{BKbook} we can find a bracketing $[u]$ of $u = \bar{f}$ that extends the regular bracketing $(\bar{g})$ of $\bar{g}$, and $\bar{[u]} = u$. By substituting $(\bar{g})$ with $g$ in $[u]$ we obtain an element $f_{\ast}$ of $\mf{g}$.  The \textit{second-order composition} $(f,g)^{II}_u$ of $f$ and $g$ is 
\[(f,g)^{II}_u = f - f_{\ast}.\]

A subset $S \subset \mf{g}$ is \textit{reduced} if all its elements are monic (that is, the monomial associated to $\bar{s}$ has coefficient $1$ for all $s \in S$) and if no second-order composition can be formed between two of its elements. 

Finally, a \textit{Gröbner-Shirshov basis} of an ideal $I \subset \mf{g}$ is a reduced set $S \subset \mf{g}$ such that 
$I = \langle \langle S \rangle \rangle$ and such that if $f, g \in S$ define a first-order composition with respect to some word $u$, then
\[ (f,g)^I_u = \sum_{i=1}^m f_i,\]
where each $f_i$ lies in $\langle \langle s_i \rangle \rangle$ for some $s_i \in S$ and $\bar{f}_i \prec u$ for $1 \leq i \leq m$.

\begin{theorem}\cite[Lemma~3.2.7]{BKbook} \label{thmShirshov}
 Let $S$ be a Gröbner-Shirshov basis of the ideal $I \subset \mf{g}$ and let $f \in \mf{g}$. If $f \in I$, then
 $\bar{f}$ contains $\bar{s}$ as a subword for some $s \in S$.
\end{theorem}

\begin{rem}
 In \cite{BKbook} (and in Shirshov's article \cite{Shirshov2})  it is assumed that the monomial ordering is the usual deg-lex, that is, $d(x_i) = 1$ for all $i$. The proofs, however, carry to our setting because the weighted deg-lex ordering is \textit{admissible}, i.e. $u \preceq v$ implies $aub \preceq avb$ for all $a,b$, and has the descending chain condition.
\end{rem}

If $S \subset \mf{g}$ is a finite homogeneous set (in the sense of the degree function $d$ defined above), then we can decide if a certain fixed element $f \in \mf{g}$ belongs to the ideal generated by $S$ as follows. First, multiplying by elements of $K$ we can assume that $S$ is monic. Then we \textit{reduce} $S$, that is, if $s_1,s_2\in S$ and $\bar{s}_1$ is a subword of $\bar{s}_2$, then we can find some $s_0 \in \langle \langle s_1\rangle \rangle$ with $\bar{s}_0 =\bar{s}_2$ and then substitute $s_2$ with $s_2 - s_0$ in $S$. This decreases the element in the weighted deg-lex ordering, so after finitely many steps we reach a reduced set $S'$ that generates the same ideal as $S$.

Now we \textit{complete} $S'$ in the following sense: we take all compositions between elements of $S'$, add them to the generating set, and then reduce again as in the previous paragraph. We repeat this process until we reach a reduced set $S^{(n)}$ such that all compositions between elements of $S^{(n)}$ are
either an element of $S^{(n)}$, or of degree greater than the degree of $\bar{f}$. 

If $\bar{f}$ does not contain $\bar{s}$ as a subword for some $s \in S^{(n)}$, then $f \notin \langle \langle S^{(n)}\rangle \rangle$. Otherwise we can find some $f_2 \in \langle \langle S^{(n)} \rangle \rangle$ such that $\bar{f}_2 = \bar{f}$, and the problem is reduced to deciding if $f-f_2$ (which is smaller with respect to weighted deg-lex) lies in this ideal. 
\end{subsection}
  
\end{section}

\begin{section}{A presentation for \texorpdfstring{$L(\mf{g})$}{L(g)}}  \label{presL}
We recall that for any Lie algebra $\mf{g}$ the weak commutativity construction is defined as
\[ \chi(\mf{g}) = \langle \mf{g}, \mf{g}^{\psi} |\ \ [x,x^{\psi}]= 0 \hbox{ for all } x \in \mf{g} \rangle,\]
for $\mf{g}^{\psi}$ an isomorphic copy of $\mf{g}$ via $x \mapsto x^{\psi}$.
The ideal $L(\mf{g})$ is the kernel of the homomorphism $\alpha: \chi(\mf{g}) \to \mf{g}$ defined by $\alpha(x)=\alpha(x^{\psi}) = x$ for all $x \in\mf{g}$.

 Let $\mf{g}$ be the free Lie algebra with free basis $x_1, \ldots, x_m$. In this section we deduce a presentation for the Lie subalgebra $L(\mf{g}) \subseteq \chi(\mf{g})$. The following lemma will be used often.
 
 \begin{lemma} \label{3id}
  Let $u,v,w \in \mf{g}$. Then:
  \[ [u-u^{\psi}, v-v^{\psi}, w-w^{\psi}] = [u,v,w] - [u,v,w]^{\psi}.\]
 \end{lemma}

 \begin{proof}
  Recall that the following identities hold in $\chi(\mf{g})$:
  \[ [x,y^{\psi}] = [x^{\psi},y]\]
  for all $x,y \in\mf{g}$. This can be deduced from the defining relation $[x-y,(x-y)^{\psi}]=0$. Furthermore, we have
  \[ [x-x^{\psi}, [y,z^{\psi}]]=0\]
  for all $x,y,z \in \mf{g}$, which is deduced as in Lemma 3.2 in \cite{Men}. It suffices now to compute the brackets on the left-hand side of the equation and apply these identities to obtain the desired result.
 \end{proof}

 First, let us find a nice generating set for $L = L(\mf{g})$. Recall that $L$ is generated as a Lie algebra by the elements $u-u^{\psi}$, with 
 $u \in \mf{g}$  (\cite{Men}, Lemma 3.1). By linearity we can further assume that $u$ is a (right-normed) bracket involving the generators $x_1, \ldots, x_m$. Using Lemma 
 \ref{3id} we can rewrite any such $u-u^{\psi}$, where $u$ has length at least $3$, as a product of terms of smaller length. More concretely, if $u = [x_{i_1}, \ldots, x_{i_n}]$
 and $n \geq 3$, then:
\[
 u-u^{\psi}=  
     \begin{cases}
       [x_{i_1}-x_{i_1}^{\psi}, \ldots, x_{i_n}-x_{i_n}^{\psi}] &\quad\text{if n is odd},\\
       [x_{i_1}-x_{i_1}^{\psi}, \ldots, x_{i_{n-2}}-x_{i_{n-2}}^{\psi}, [x_{i_{n-1}},x_{i_n}] - [x_{i_{n-1}},x_{i_n}]^{\psi}]  &\quad\text{otherwise.} \\ 
     \end{cases}
\]
This tells us that the elements $x_i-x_i^{\psi}$ and $[x_i,x_j]-[x_i,x_j]^{\psi}$ for $1 \leq i < j \leq m$ generate $L$ as a Lie algebra. 

We introduce the notation:
\[ \tilde{a}_i := x_i-x_i^{\psi}\]
for $1 \leq i \leq m$ and 
\[ \tilde{b}_{i,j} := [x_i,x_j] - [x_i,x_j]^{\psi}\]
for $1 \leq i < j \leq m$. For convenience we set $\tilde{b}_{i,i}=0$ and $\tilde{b}_{i,j} = - \tilde{b}_{j,i}$ for all $1 \leq j <  i \leq m$. 
 
In the next two lemmas we deduce some relations among these generators. They all come from suitable interpretations of the identity of Lemma \ref{3id}. We
will show later that they actually form a full set of relations for $L$ with the generators that we chose.

\begin{lemma} \label{relL1}
For any $i,j,k,l$, we have:
\[[\tilde{a}_i, \tilde{a}_j, \tilde{b}_{k,l}] = [\tilde{a}_i, \tilde{b}_{j,k}, \tilde{a}_l] + [\tilde{a}_i, \tilde{a}_k, \tilde{b}_{j,l}]\]
\end{lemma}

\begin{proof}
We use Lemma \ref{3id}, the Jacobi identity and then Lemma \ref{3id} again:
\[
  \begin{array}{lcl} 
[\tilde{a}_i, \tilde{a}_j, \tilde{b}_{k,l}]  & = & [x_i-x_i^{\psi}, x_j-x_j^{\psi}, [x_k,x_l]-[x_k,x_l]^{\psi}]\\
   & = &  [x_i,x_j,x_k,x_l] - [x_i,x_j,x_k,x_l]^{\psi} \\
   & = & [x_i,[x_j,x_k],x_l] - [x_i,[x_j,x_k],x_l]^{\psi} + [x_i,x_k,x_j,x_l] - [x_i,x_k,x_j,x_l]^{\psi}  \\
   & = & [\tilde{a}_i,\tilde{b}_{j,k},\tilde{a}_l] + [\tilde{a}_i,\tilde{a}_k,\tilde{b}_{j,l}],
\end{array} 
\]
as we wanted.
\end{proof}

\begin{lemma}  \label{relL2}
 Let $n > 1$ be an odd integer and let $u_i$ be either some $\tilde{a}_j$ or some $\tilde{b}_{j,j'}$, for all $i$. Then:
\[[u_1, \ldots, \tilde{b}_{j,j'}, \ldots, \tilde{b}_{k,k'}, \ldots, u_n] = [u_1, \ldots, [\tilde{a}_j,\tilde{a}_{j'}], \ldots, [\tilde{a}_k,\tilde{a}_{k'}], \ldots, u_n]\]
and
\[ [u_1, \ldots, \tilde{b}_{j,j'}, \ldots, u_{n-2}, \tilde{a}_{n-1}, \tilde{a}_n] = [u_1, \ldots, [\tilde{a}_j,\tilde{a}_{j'}], \ldots, u_{n-2}, \tilde{b}_{n-1,n}].\] 
 \end{lemma}
 
\begin{proof}
Notice that Lemma \ref{3id} actually gives a way of writing any right-normed bracket of odd length involving elements of the form $v-v^{\psi}$ as a single
element of this same form. It suffices then to apply this reasoning to each of the brackets in the statement of the lemma and verify that the two sides
of each equation actually reduce to the same element.
\end{proof}

Define $\mc{L}$ as the Lie algebra generated by the symbols $a_i$ and $b_{i,j}$, for $1 \leq i < j \leq m$, subject to the relators as in Lemmas 
\ref{relL1} and \ref{relL2} (without the tildes), that is,
\begin{equation}  \label{Lrel1}
r = [a_i, a_j, b_{k,l}] - [a_i, b_{j,k}, a_l] - [a_i, a_k, b_{j,l}]
\end{equation}
for all $i,j,k,l$,
\begin{equation} \label{Lrel2.1}
r = [u_1, \ldots, b_{j,j'}, \ldots, b_{k,k'}, \ldots, u_n] - [u_1, \ldots, [a_j,a_{j'}], \ldots, [a_k,a_{k'}], \ldots, u_n]
\end{equation}
for all $u_i \in \{a_t, b_{s,t}\}_{s,t}$, all indices and $n\geq 3$ an odd integer and
\begin{equation} \label{Lrel2.2}
r = [u_1, \ldots, b_{j,j'}, \ldots, u_{n-2}, a_{n-1}, a_n] - [u_1, \ldots, [a_j,a_{j'}], \ldots, u_{n-2}, b_{n-1,n}],
\end{equation}
for all $u_i \in \{a_t, b_{s,t}\}_{s,t}$, all indices and $n\geq 3$ an odd integer. Again we are using the convention that $b_{i,j} = - b_{j,i}$ and $b_{i,i}=0$, so that
all expressions above make sense. We will show that $\mc{L} \simeq L(\mf{g})$.

\begin{lemma}  \label{actionbrac}
 The following formulas define an action of $\mf{g}$ on $\mc{L}$:
 \[ x_s \cdot a_i = \frac{1}{2} ( [a_s,a_i] + b_{s,i})\]
 and 
 \[ x_s \cdot b_{i,j} = \frac{1}{2} ( [a_s,b_{i,j}] + [a_s, a_i, a_j]),\]
 for all $s,i,j$.
\end{lemma}

\begin{proof}
Let $F$ be the free Lie algebra on $a_i$, $b_{i,j}$, for $1 \leq i < j \leq m$. It is clear that the formulas above give well-defined derivations 
$D_s: F \to F$. To see that they induce derivations $D_s: \mc{L} \to \mc{L}$ we need to verify that the relations are respected.

First notice that
\[ D_s( [a_i, a_j, b_{k,l}]) = \frac{1}{2} ([a_s, a_i, a_j, b_{k,l}] + [b_{s,i},a_j,b_{k,l}] + [a_i, b_{s,j}, b_{k,l}] + [a_i, a_j, a_s, a_k, a_l]).\]
Modulo relators of type \eqref{Lrel2.1} we have
\[ [b_{s,i},a_j,b_{k,l}] = [ [a_s,a_i],a_j, a_k,a_l]\]
and
\[[a_i, b_{s,j}, b_{k,l}] = [ a_i, [a_s,a_j], a_k,a_l],\]
thus
\[ D_s( [a_i, a_j, b_{k,l}]) = \frac{1}{2} ([a_s, a_i, a_j, b_{k,l}] + [a_s, a_i,a_j, a_k, a_l]).\]
It becomes clear from this expression that for 
\[ r = [a_i, a_j, b_{k,l}] - [a_i, b_{j,k}, a_l] - [a_i, a_k, b_{j,l}] \in F,\]
the image $D_s(r)$ is a consequence of the defining relators of $\mc{L}$.

Now let 
\[r_{21} = [u_1, \ldots, b_{j,j'}, \ldots, b_{k,k'}, \ldots, u_n]\]
\[r_{22} = [u_1, \ldots, [a_j,a_{j'}], \ldots, [a_k,a_{k'}], \ldots, u_n],\]
with $n>1$ odd. We need to verify that $D_s( r_{21}-r_{22})$ is a consequence of the defining relators of $\mc{L}$. 
We compute each $D_s(r_{2i})$ by applying the derivation property. Notice that
\begin{dmath*}
 D_s(r_{21}) - \frac{1}{2}[a_s, r_{21}] = \frac{1}{2}([u_1, \ldots, [a_s, a_j,a_{j'}], \ldots, b_{k,k'}, \ldots, u_n] + [u_1, \ldots, b_{j,j'}, \ldots, [a_s, a_k, a_{k'}], \ldots, u_n] + A),
 \end{dmath*}
where $A$ is the sum of the terms \[[u_1, \ldots, D_s(u_i)-[a_s,u_i], \ldots, b_{j,j'}, \ldots, b_{k,k'}, \ldots, u_n]\] with $i \neq j, k$.
Similarly we have:
\begin{dmath*}
 D_s(r_{22}) - \frac{1}{2}[a_s, r_{22}] = \frac{1}{2}([u_1, \ldots, [b_{s,j},a_{j'}] + [a_j,b_{s,j'}], \ldots, [a_k,a_k'], \ldots, u_n] + [u_1, \ldots, [a_j,a_{j'}], \ldots, [b_{s,k}, a_{k'}]+ [a_k,b_{s,k'}], \ldots, u_n] + B),
\end{dmath*}
where $B$ is the sum of the terms \[[u_1, \ldots, D_s(u_i)-[a_s,u_i], \ldots [a_j,a_{j'}], \ldots, [a_k,a_{k'}], \ldots, u_n]\] with $i \neq j, k$. 

Now notice that each of the terms of $D_s(r_{22}) - \frac{1}{2}[a_s, r_{22}]$ can be transformed into the respective term of $D_s(r_{21}) - \frac{1}{2}[a_s, r_{21}]$
modulo a relator of type \eqref{Lrel2.1} or \eqref{Lrel2.2} (or that after an application of the Jacobi identity). We recall that $n$ is an odd integer to begin with, 
so all  
brackets have length $n$ or $n+2$ and the application of the defining relators is licit. Thus
\[ D_s( r_{21}-r_{22}) \equiv \frac{1}{2}[a_s, r_{21}-r_{22}] \equiv 0,\]
where the congruence is taken modulo the defining relators of $\mc{L}$. The proof for a relator of type \eqref{Lrel2.2} is completely analogous,
so we omit.

Finally, since each $D_s: \mc{L} \to \mc{L}$ is a well defined derivation and $\mf{g}$ is free, the association $x_s \mapsto D_s$ defines a Lie algebra
homomorphism $\mf{g} \to Der(\mc{L})$, that is, an action of $\mf{g}$ on $\mc{L}$.
\end{proof}

We will abandon the notation $D_s$ for these derivations; we will simply write $x_s \cdot \ell$ to denote the action of $D_s$ on some $\ell \in \mc{L}$.
We may consider the 
semi-direct product $\mc{L} \rtimes \mf{g}$ defined with respect to this action, so that $x_s \cdot \ell$ is identified with $[x_s,\ell]$ for all $\ell \in \mc{L}$.

In order to prove that $\mc{L} \simeq L(\mf{g})$, we first deduce some formulas for the action of $\mf{g}$ on $\mc{L}$. For instance, we have this nice 
formula for the action of long right-normed brackets involving the generators of $\mf{g}$ on the generators of $\mc{L}$.
\begin{lemma}  \label{L35}
For all $n \geq 2$ and for any $u \in \{a_i, b_{i,j}\}$ we have:
\[ [x_{i_1}, \ldots, x_{i_n}] \cdot u = -\frac{1}{2} ( [u, a_{i_1}, \ldots, a_{i_n}] + [u, a_{i_1}, \ldots, a_{i_{n-2}}, b_{i_{n-1},i_n}])\]
\end{lemma}

\begin{proof}
This can be proved by induction on $n$. The formulas are easily verified for $n=2$. For $n>2$ we use the following fact: if $\alpha$ and $\beta$ are derivations of a Lie algebra $L$, and $\beta(x) = [x,b]$ for some $b \in L$ and for all $x \in L$ (that is, $\beta$ is inner), then $\alpha(x) = [x,\alpha(b)]$ for all $x \in L$.

In order to use the induction hypothesis, we write
\[ [x_{i_1}, \ldots, x_{i_n}] = [x_{i_1}, [x_{i_2}, \ldots, x_{i_n}]].\]
The derivation associated to  $[x_{i_2}, \ldots, x_{i_n}]$ is inner and the multiplying element is given by the statement. Thus 
\[ [x_{i_1}, \ldots, x_{i_n}] \cdot u =[u, x_{i_1} \cdot ([a_{i_2}, \ldots, a_{i_n}] + [a_{i_2}, \ldots, a_{i_{n-2}}, b_{i_{n-1},i_n}])]
   \]
for all $u$.

By computing via the derivation property, we actually obtain
\[x_{i_1} \cdot ([a_{i_2}, \ldots, a_{i_n}] + [a_{i_2}, \ldots, a_{i_{n-2}}, b_{i_{n-1},i_n}]) = [a_{i_1}, \ldots, a_{i_n}] + [a_{i_1}, \ldots, a_{i_{n-2}}, b_{i_{n-1},i_n}] + s\]
modulo the defining relators of $\mc{L}$, where $s$ in an element such that $[u,s]$ is a defining relator of $\mc{L}$ for all $u \in \{a_i, b_{i,j}\}$ (in other words, $s$ is central). This gives the desired result.
\end{proof}


Analogously to Lemma \ref{3id}, we have a nice way of writing some long brackets of $\mc{L}$.

\begin{lemma}  \label{L36}
 We have:
 \[
 [x_{i_1}-a_{i_1}, \ldots, x_{i_n}-a_{i_n}]=  
     \begin{cases}
       [x_{i_1}, \ldots, x_{i_n}] - [a_{i_1}, \ldots, a_{i_n}] &\quad\text{if n is odd},\\
       [x_{i_1}, \ldots, x_{i_n}] - [a_{i_1}, \ldots, a_{i_{n-2}}, b_{i_{n-1}, i_n}]   &\quad\text{otherwise.} \\ 
     \end{cases}
\]
\end{lemma}

\begin{proof}
 We prove it by induction on $n$. If $n=1$ this is clear. If $n>1$ is even, then by induction hypothesis we have
 \[ [x_{i_1}-a_{i_1}, \ldots, x_{i_n}-a_{i_n}]  = [x_{i_1}-a_{i_1}, [x_{i_2}, \ldots, x_{i_n}] - [a_{i_2}, \ldots, a_{i_n}]].\]
By Lemma \ref{L35} we have
\begin{equation}  \label{T1}
 -[a_{i_1}, [x_{i_2}, \ldots, x_{i_n}]] = -\frac{1}{2}([a_{i_1}, \ldots, a_{i_n}] +  [a_{i_1}, \ldots, a_{i_{n-2}}, b_{i_{n-1}, i_n}]).
\end{equation}
Also:
\begin{equation}  \label{T2}
 -[x_{i_1}, [a_{i_2}, \ldots, a_{i_n}]] = -\frac{1}{2}([a_{i_1}, \ldots, a_{i_n}] + \sum_{j=2}^n[a_{i_2}, \ldots, b_{i_1,i_j}, \ldots, a_{i_n}]). 
\end{equation}
By relations \ref{Lrel1} and \ref{Lrel2.2} it follows that 
\[\sum_{j=2}^n[a_{i_2}, \ldots, b_{i_1,i_j}, \ldots, a_{i_n}] = [a_{i_1}, a_{i_2}, \ldots, a_{i_{n-2}}, b_{i_{n-1}, i_n}].\]
Finally, by summing \eqref{T1} and \eqref{T2}, we get
\[ [x_{i_1}-a_{i_1}, \ldots, x_{i_n}-a_{i_n}] = [x_{i_1}, \ldots, x_{i_n}] - [a_{i_1}, \ldots, a_{i_{n-2}}, b_{i_{n-1}, i_n}], \]
as we wanted. 

If $n$ is odd we have similarly
 \[ [x_{i_1}-a_{i_1}, \ldots, x_{i_n}-a_{i_n}]  = [x_{i_1}-a_{i_1}, [x_{i_2}, \ldots, x_{i_n}] - [a_{i_2}, \ldots, a_{i_{n-2}}, b_{i_{n-1},i_n}]].\]
This time we have:
\begin{dmath}  \label{T3}  
           -[x_{i_1}, [a_{i_2}, \ldots, a_{i_{n-2}}, b_{i_{n-1},i_n}]] = -\frac{1}{2}( [a_{i_1}, \ldots, a_{i_{n-2}}, b_{i_{n-1}, i_n}] + 
           \sum_{j=1}^{n-2}[a_{i_2}, \ldots, b_{i_1, i_j}, \ldots a_{i_{n-2}}, b_{i_{n-1},i_n}]  + [a_{i_2}, \ldots, a_{i_{n-2}}, a_{i_1}, a_{i_{n-1}}, a_{i_n}])
\end{dmath}
Now, by relations \eqref{Lrel2.1} we get
\[
\sum_{j=1}^{n-2}[a_{i_2}, \ldots, b_{i_1, i_j}, \ldots a_{i_{n-2}}, b_{i_{n-1},i_n}]  + [a_{i_2}, \ldots, a_{i_{n-2}}, a_{i_1}, a_{i_{n-1}}, a_{i_n}] = [a_{i_1}, \ldots, a_{i_n}].
\]
Thus by \eqref{T1} and \eqref{T3} we get the result. The proof is complete.
\end{proof}

 Let $\mc{M}$ be the set of right-normed brackets involving $x_1, \ldots, x_m$.  For $u = [x_{i_1}, \ldots, x_{i_n}] \in \mc{M}$, with $n \geq 2$,
 denote
 \[ \mu(u) =  
     \begin{cases}
       [a_{i_1}, \ldots, a_{i_n}] &\quad\text{if n is odd},\\
       [a_{i_1}, \ldots, a_{i_{n-2}}, b_{i_{n-1}, i_n}]   &\quad\text{otherwise.} \\ 
     \end{cases}
 \]
 Similarly, let:
  \[ \xi(u) =  
     \begin{cases}
       [a_{i_1}, \ldots, a_{i_{n-2}}, b_{i_{n-1}, i_n}]   &\quad\text{if n is odd},\\
       [a_{i_1}, \ldots, a_{i_n}] &\quad\text{otherwise.} \\ 
     \end{cases}
 \]
 By Lemma \ref{L35}, for $u = [x_{i_1}, \ldots, x_{i_n}]$ with $n \geq 2$, we have:
 \[ u \cdot \ell = -\frac{1}{2}[\ell, \mu(u) + \xi(u)]\]
 for all $\ell \in \mc{L}$.

\begin{rem}  \label{relchi}
 If $\mf{g}$ is free on $x_1, \ldots, x_m$, then by linearity a full set of relations for $\chi(\mf{g})$ is given by:
 \[ [u,v^{\psi}] = [u^{\psi}, v]\]
 for all right-normed brackets $u,v$ involving the generators $x_1, \ldots, x_m$ (including the case $u=v$).
\end{rem} 
 
\begin{theorem}  \label{Thm1}
Let $\mf{g}$ be free on the set $\{x_1, \ldots, x_m\}$. Then $\chi(\mf{g}) \simeq \mc{L} \rtimes \mf{g}$.
\end{theorem}

\begin{proof}
 Define \[ \sigma: \mc{L} \rtimes \mf{g} \to \chi(\mf{g})\]
 by \[ \sigma(x_i) = x_i, \hbox{   } \sigma(a_i) = x_i-x_i^{\psi}, \hbox{   } \sigma(b_{i,j}) = [x_i,x_j]-[x_i,x_j]^{\psi}.\]
 By choice, this is a well-defined surjective homomorphism of Lie algebras.
 
 Similarly, define
 \[ \theta: \chi(\mf{g}) \to \mc{L} \rtimes \mf{g}\]
 by
 \[ \theta(x_i) = x_i, \hbox{  } \theta(x_i^{\psi}) = x_i-a_i.\]
 We need show that $\theta$ is well-defined. We will verify that the relations as in Remark \ref{relchi} are preserved.
 
 Let $u,v \in \mc{M}$. If $u$ and $v$ have length $1$, say $u=x_i$ and $v=x_j$, we can see directly by the definition of the action of $\mf{g}$ on
 $\mc{L}$ that $[x_i, a_j] = -[a_i, x_j]$, which implies that $[\theta(u), \theta(v^{\psi})] = [\theta(u^{\psi}), \theta(v)]$.
 
 So suppose that $v$ has length at least $2$. By Lemma \ref{L36} we have
 \[ [\theta(u), \theta(v^{\psi})] = [u, v- \mu(v)] = [u,v] - [u,\mu(v)].\]
 If $u=x_t$ has length $1$, this reduces to $[x_t, v]- [x_t, \mu(v)]$ and we have on the other hand:
 \[ [\theta(u^{\psi}), \theta(v)] = [x_t-a_t, v] = [x_t,v] - [a_t, v].\]
 Now write $v = [x_{i_1}, \ldots, x_{i_n}]$. If $n$ is odd, we have:
 \[ [x_t, \mu(v)] = [x_t, [a_{i_1}, \ldots, a_{i_n}]] = \frac{1}{2}\sum_{j=1}^n [a_{i_1}, \ldots, [a_t, a_{i_j}]+ b_{t,i_j}, \ldots, a_{i_n}].\]  
 For $1 \leq j \leq n-2$, as a consequence of relation \eqref{Lrel2.2} we have
 \[ [a_{i_1}, \ldots, b_{t,i_j}, \ldots, a_{i_n}] = [a_{i_1}, \ldots, [a_t,a_{i_j}], \ldots, a_{i_{n-2}}, b_{i_{n-1}, i_n}].\]
 Also, by relation \eqref{Lrel1} we have:
 \[ [a_{i_1}, \ldots, a_{i_{n-2}},  b_{t,i_{n-1}}, a_{i_n}] + [a_{i_1}, \ldots, a_{i_{n-2}}, a_{i_{n-1}}, b_{t,i_n}] = [a_{i_1}, \ldots, a_{i_{n-2}}, a_t, b_{i_{n-1}, i_n}]\]
 Thus:
 \[ [x_t, \mu(v)] = \frac{1}{2}([a_t, a_{i_1}, \ldots, a_{i_n}] + [a_t, a_{i_1}, \ldots, a_{i_{n-2}}. b_{i_{n-1},i_n}]) = \frac{1}{2}[a_t, \mu(v) + \xi(v)]\]
 This coincides with the formula given on Lemma \ref{L35} for $[a_t,v]$, so $[\theta(x_i), \theta(v^{\psi})] = [\theta(x_i^{\psi}), \theta(v)]$.
 The same reasoning gives the result if $n$ is even, but we use the relations \eqref{Lrel2.1}.
 
 Finally, suppose that both $u$ and $v$ have length at least two. Since $\mu(v) \in \mc{L}$, we have by the comment above the theorem:
 \[ [\theta(u), \theta(v)^{\psi}] = [u,v] - [u,\mu(v)] = [u,v] + \frac{1}{2}[\mu(v), \mu(u)+ \xi(u)].\]
 By looking similarly at $[\theta(u^{\psi}), \theta(v)]$, we see that in order to show that the relation $[u,v^{\psi}]=[u^{\psi}, v]$ is preserved by
 $\theta$, we need only to verify that 
 \[ [\mu(u), \xi(v)] = [\xi(u), \mu(v)]\]
 for all $u,v$. But this is an instance of relation \eqref{Lrel2.2}, after opening up the brackets as right-normed ones (note that $[\mu(u),\xi(v)]$ is
 always a bracket of odd length). 
 
 Thus $\theta$ is well defined and clearly $\sigma \circ \theta = id$ and $\theta \circ \sigma = id$, that is, $\theta$ is an isomorphism of Lie algebras.
\end{proof}

 By restriction we get an isomorphism $L(\mf{g}) \simeq \mc{L}$.
 
 In general, if $\mf{h} = \mf{g}/N$, then $\chi(\mf{g})$ is the quotient of $\mc{L} \rtimes \mf{g}$ by the ideal generated by $\theta(N \cup N^{\psi})$, where $\theta$ is the homomorphism defined in the above proof. Clearly $\theta(N)$ generates a copy of $N$ inside $\mf{g}$.  On the other hand, if $r \in N$, then 
 \[\theta(r^{\psi}) = r - \mu(r),\]
 where $\mu$ is extended by linearity (or $\mu: \mf{g} \to \mc{L} \rtimes \mf{g}$ is \textit{defined} as $\mu = inc - \theta |_{\mf{g}^{\psi}}$, where $inc: \mf{g} \to \mc{L} \rtimes \mf{g}$ is the obvious inclusion). 
 It follows that $\chi(\mf{h}) = \mc{L}/J \rtimes \mf{h}$, where $J$ is the ideal of $\chi(\mf{g})$ generated by $\mu(r)$, for all $r \in N$. This gives implicitly a presentation for $L(\mf{h})$.

 This presentation of $L(\mf{g})$ is of course infinite. For a non-abelian free Lie algebra $\mf{g}$ in fact $L(\mf{g})$ does not admit a finite presentation.
 
 \begin{prop} \label{H2Linfinite}
If $\mf{g}$ is free non-abelian, then $L(\mf{g})$ does not admit a finite presentation.
\end{prop}

\begin{proof}
 It suffices to show that $H_2(L(\mf{g});K)$ is infinite-dimensional. Suppose, on the contrary, that it is of finite dimension. Recall that $W/R \simeq H_2(\mf{g};K)$. In particular $W = R$ if $\mf{g}$ is free. By analyzing $R$ in terms of the generators of $\mc{L}$, it becomes clear that $R \subseteq [L,L]$. In particular, $L^{ab} \simeq \rho(L)^{ab}$. Now, the $5$-term exact sequence associated to the LHS spectral sequence arising from $R \rightarrowtail L \twoheadrightarrow \rho(L)$ reduces to
 \[ H_2(L;K) \to H_2(\rho(L); K) \to R \to 0.\]
 By \cite[Lemma~5.5]{Men}, the homology $H_2(\rho(L); K)$ is infinite-dimensional. Thus $R$ is infinite-dimensional (we are under the hypothesis that $H_2(L;K)$ is finite-dimensional).
 
 Now, consider the spectral sequence itself
 \[E_{p,q}^2 = H_p(\rho(L); H_q(R;K)) \Rightarrow H_{p+q}(L;K).\]
 Notice that $E_{1,1}^2= E_{1,1}^{\infty}$. Indeed, the differentials involved are $d_{1,1}: E_{1,1}^2 \to E_{-1,2}^2$ 
 and $d_{1,1}: E_{3,0}^2 \to E_{1,1}^2$. Clearly $E_{-1,2}^2=0$, but also $E_{3,0}^2 = H_3(\rho(L);K) = 0$, since 
 $\rho(L) \subseteq \mf{g} \oplus \mf{g}$ (thus $cd(\rho(L)) \leq 2$). Then $E_{1,1}^2=E_{1,1}^3= E_{1,1}^{\infty}$.
 But:
 \[ E_{1,1}^2= H_1( \rho(L); H_1(R;K)) \simeq \rho(L)^{ab} \otimes_K R,\]
 since $\rho(L)$ acts trivially on $R$. Thus, if $R$ is infinite-dimensional, then so is $E_{1,1}^{\infty}$, and finally so is $H_2(L;K)$. This is a contradiction. 
 \end{proof}

 \begin{rem}
  Notice that if $\mf{g}$ is free of rank $2$ the conclusion was more immediate: $W = R = 0$, so $L \simeq \rho(L)$, and
  we already knew that $H_2( \rho(L);K)$ was infinite-dimensional.  
 \end{rem}

 \end{section}

\begin{section}{Nilpotent Lie algebras} \label{sec1}
 If $\mf{g}$ is abelian, then $\chi(\mf{g})$ is completely described in Proposition 7.1 in \cite{Men}. We consider here nilpotent Lie algebras 
 of class $c \geq 2$. We will show that if $\mf{g}$ is nilpotent of class $c$, then $\chi(\mf{g})$ is nilpotent of class at most $c+2$.

Denote by $\mf{n}_{m,c}$ the free nilpotent Lie algebra of rank $m$ and class $c$. By the comments in the previous section, we can obtain $L(\mf{n}_{m,c})$ by taking the quotient of $\mc{L} \simeq L(\mf{g})$ (where $\mf{g}$ is free of rank $m$) by the ideal generated by the elements $\mu(u)$, for brackets $u$ of length at least $c+1$ involving the generators of $\mf{g}$.

Consider the generators $a_i$, $b_{i,j}$ of $\mc{L}$. Define
\[ d(a_i) := 1, \hbox{  } d(b_{i,j}) := 2,\]
for all $i<j$. For a right-normed bracket $\ell = [\ell_1, \ldots, \ell_n]$ involving some of these generators, we define the \textit{degree} $d(\ell)$
of $\ell$ as
\[ d(\ell) = \sum_{j=1}^n d(\ell_j).\]

Now we are ready to determine the (class of) nilpotency of $\chi(\mf{g})$.

\begin{theorem} \label{thmnilpclass}
Suppose that $\mf{h}$ is nilpotent of class $c$. Then $\chi(\mf{h})$ is nilpotent and its nilpotency class is bounded by the smallest even integer greater than $c$.
\end{theorem}

\begin{proof}
 Clearly we can assume that $\mf{h}$ is free nilpotent of class $c$, so that all the comments above the theorem make sense. Suppose that we want to 
 show that $\chi(\mf{h})$ is nilpotent of class $n$, where $n$ may be $c+1$ or $c+2$ depending on the parity of $c$. Let $x_1, \ldots, x_m$ be a set 
 of generators for $\mf{h}$. It is enough then to show that 
 \[ w=[x_{i_{n+1}}^{\theta_{n+1}}, \ldots, x_{i_2}^{\theta_2}, x_{i_1}^{\theta_1}]=0\]
 for all $1 \leq i_j \leq m$ and $\theta_j \in \{id, \psi\}$. Clearly if $\theta_i=id$ for all $i$, then $w=0$. Similarly, $w=0$ if $\theta_i = \psi$ for 
 all $i$. We can assume without loss of generality that $\theta_1 = id$.  Let $k = min \{j | \theta_j \neq id\}$. Since 
 $[x_{i_k}^{\psi}, x_{i_{k-1}}, \ldots, x_{i_1}] \in D$ and $[D,L]=0$, we have
 \[ w=[x_{i_{n+1}}, \ldots, x_{i_{k+1}},x_{i_k}^{\psi}, x_{i_{k-1}}, \ldots, x_{i_2}, x_{i_1}].\]
 It follows by induction on $k$ that $w$ is a linear combination of terms of the form $[x_{j_{n+1}}, \ldots, x_{j_2}, x_{j_1}^{\psi}]$, for some
 $1 \leq j_t\leq m$. Indeed, this is clear if $k=2$. If $k>2$, by the Jacobi identity we have:
\begin{dmath}
   w=[x_{i_{n+1}}, \ldots, x_{i_{k+1}}, x_{i_{k-1}}, x_{i_k}^{\psi}, x_{i_{k-2}}, \ldots, x_{i_2}, x_{i_1}]+
   [x_{i_{n+1}}, \ldots, x_{i_{k+1}}, [x_{i_k}^{\psi}, x_{i_{k-1}}], x_{i_{k-2}}, \ldots,x_{i_2}, x_{i_1}].
  \end{dmath}
 By induction hypothesis we can rewrite the first term on the right-hand side of the equation above in the form we want. By antisymmetry the second term is
 \[[x_{i_{n+1}}, \ldots,x_{i_{k+1}}, [x_{i_{k-2}}, \ldots,x_{i_2}, x_{i_1}],  [x_{i_{k-1}},x_{i_k}^{\psi}]],\]
 which can be rewritten by the Jacobi identity as a linear combination of terms of the form
 \[ [x_{i_{n+1}}, \ldots, x_{i_{k+1}}, x_{\sigma(i_{k-2})}, \ldots, x_{\sigma(i_2)}, x_{\sigma(i_1)}, x_{i_{k-1}}, x_{i_k}^{\psi}] \]
 for some permutations $\sigma \in S_{k-2}$. All of this means that $\chi(\mf{h})$ is nilpotent of class $n$ if
\[ w=[x_{i_{n+1}}, \ldots, x_{i_2}, x_{i_1}^{\psi}]=0\]
for all $1 \leq i \leq m$. 

Now we interpret this in terms of the isomorphism $\theta: \chi(\mf{h}) \to \mc{L}/J \rtimes \mf{h}$. We have:
\[ \theta(w) = \theta([x_{i_{n+1}}, \ldots, x_{i_2}, x_{i_1}^{\psi}]) = [x_{i_{n+1}}, \ldots, x_{i_2}, x_{i_1}-a_{i_1}]= -[x_{i_{n+1}}, \ldots, x_{i_2}, a_{i_1}],\]
since $[x_{i_{n+1}}, \ldots, x_{i_2}, x_{i_1}] = 0$. By induction we see that $-[x_{i_{n+1}}, \ldots, x_{i_2}, a_{i_1}]$ is a linear 
combination of brackets $\ell = [\ell_1, \ldots, \ell_k]$, involving the generators of $\mc{L}$, with $d(\ell)=n+1$. 

Finally, we consider the parity of $c$. If $c$ is odd, we are trying to prove that $\chi(\mf{g})$ is nilpotent of class $c+1$, that is, $n=c+1$. Given a bracket 
$\ell = [\ell_1, \ldots, \ell_k]$ with $d(\ell) = c+2$, we can use the defining relations of $\mc{L}$ to rewrite it as a linear combination of elements of the forms
\[[a_{i_1}, \ldots, a_{i_{c+2}}]\]
and
\[[a_{i_1}, \ldots, a_{i_c}, b_{i_{c+1},i_{c+2}}].\]
Notice that it essential the fact that $c+2$ is an odd integer, otherwise we would not be able to get rid of brackets of the form
\[[b_{i_1,i_2}, a_{i_3}, \ldots, a_{i_c}, b_{i_{c+1},i_{c+2}}].\]

Now, as observed before, $\mu(u)$ is trivial in $\mc{L}$ for any $u$ a bracket involving the generators of $\mf{g}$ with length at least $c+1$. In particular, 
for $u=[x_{i_1}, \ldots, x_{i_{c+2}}]$ and $v=[x_{i_2}, \ldots, x_{i_{c+2}}]$ we get
\[ \mu(u) = [a_{i_1}, \ldots, a_{i_{c+1}},a_{i_{c+2}}]\]
and
\[\mu(v)= [a_{i_2}, \ldots, a_{i_c}, b_{i_{c+1},i_{c+2}}].\]
Clearly $\mu(u)=0$ and $\mu(v)=0$ for all $u$ and $v$ of those forms implies that $\ell=0$. Thus $\chi(\mf{g})$ is nilpotent of class at most $c+1$.

Similarly, suppose that $c$ is even. Now we want to show that $\chi(\mf{g})$ is nilpotent of class at most $n=c+2$. Once again $n+1 = c+3$ is an odd integer,
so as before we only need to show that brackets of the forms
\[[a_{i_1}, \ldots, a_{i_{c+3}}]\]
and
\[[a_{i_1}, \ldots, a_{i_{c+1}}, b_{i_{c+2},i_{c+3}}].\]
The same argument works: the fact that $\mu([x_{i_1}, \ldots, x_{i_{c+1}}])$ and $\mu([x_{i_1}, \ldots, x_{i_{c+2}}])$ must be trivial in $\mc{L}$ is enough to guarantee what we
want. In this case the proof works to show that $\mf{g}$ must be nilpotent of class at most $n= c+2$, as we wanted.
\end{proof}

 These bounds are sharp in the generality of the statement of the theorem, as we will see in the next section. We can, however, obtain a sharper result for $2$-generated Lie algebras
 by a very simple argument.
\begin{prop}
If $\mf{g}$ is $2$-generated and nilpotent of class $c$, then $\chi(\mf{g})$ is nilpotent of class $c+1$.
\end{prop}

\begin{proof}
 As in the proof of the theorem, it is enough to show that right-normed brackets of the form
 \[w = [x_{i_1},  \ldots, x_{i_{c+1}}, x_{i_{c+2}}^{\psi}]\]
 are trivial. Now, by Section 7 of \cite{Men} we know that $R(\mf{g}) = [\mf{g},L(\mf{g}), \mf{g}^{\psi}]=0$ whenever $\mf{g}$ is $2$-generated. In particular,
 \[[u,v,w^{\psi}] = [u,v^{\psi},w^{\psi}]\]
 for all $u,v,w \in \mf{g}$. But then, by induction, we have:
 \[w = [x_{i_1}, \ldots, x_{i_c}, x_{i_{c+1}}^{\psi}, x_{i_{c+2}}^{\psi}] = \ldots = [x_{i_1}, x_{i_2}^{\psi}, \ldots,  x_{i_{c+1}}^{\psi}, x_{i_{c+2}}^{\psi}]=0,\]
 since $[x_{i_2}^{\psi}, \ldots,  x_{i_{c+1}}^{\psi}, x_{i_{c+2}}^{\psi}]$ is trivial in $\mf{g}^{\psi}$.
 \end{proof}
\end{section}

\begin{section}{Examples}
For the classes of nilpotency $c \leq 3$, we can actually get from the proofs in the previous sections a concrete description of $\chi(\mf{n}_{m,c})$.

\subsection{Free nilpotent of class 2}

\begin{crlr}
If $\mf{h} =  \mf{n}_{m,2}$, then $L(\mf{h})$ is free nilpotent of rank $m+ \binom{m}{2}$ and class $2$. In particular, we have
\[dim \chi(\mf{h}) = 2k+ \binom{k}{2},\]
where $k= m + \binom{m}{2}$, and $\chi(\mf{h})$ is nilpotent of class exactly $4$ if $m \geq 3$.
\end{crlr}

\begin{proof}
By the previous section, $\mu(u)$ is trivial if $u$ has length at least $3$. Thus:
\begin{equation}  \label{c2.1}
 \mu([x_{i_1}, x_{i_2}, x_{i_3}]) = [a_{i_1}, a_{i_2}, a_{i_3}]=0
\end{equation}
 and 
\begin{equation} \label{c2.2}
 \mu([x_{i_1}, x_{i_2}, x_{i_3}, x_{i_4}]) = [a_{i_1}, a_{i_2}, b_{i_3,i_4}]=0
\end{equation}
for all $i_j$. It is clear then that any $[a_i,a_j]$ is a central element in $\mc{L}$. Moreover,
by relation \eqref{Lrel2.1} we have
\[ [b_{i_1,i_2}, a_{i_3}, b_{i_4,i_5}] = [[a_{i_1},a_{i_2}], a_{i_3}, a_{i_4}, a_{i_5}]\]
and 
\[ [b_{i_1,i_2}, b_{i_3,i_4}, b_{i_5,i_6}] = [[a_{i_1},a_{i_2}], [a_{i_3},a_{i_4}], a_{i_5}, a_{i_6}]\]
which also become trivial by \eqref{c2.1}. This enough to conclude that both $[a_i,b_{j,k}]$ and $[b_{i,j}, b_{k,l}]$ are also central in $\mc{L}$.
Finally, it is clear that the original relations of $\mc{L}$ and all $\mu(u)$, with $u$ a bracket of length greater than $4$, are actually consequences of \eqref{c2.1} and \eqref{c2.2}. Thus 
$L(\mf{h})$ is free nilpotent of class $2$ with basis $a_i$ and $b_{i,j}$, for all $1 \leq i < j \leq m$.

The formula for the dimension follows clearly from the fact that $\chi(\mf{h}) \simeq L(\mf{h}) \rtimes \mf{h}$. Finally, if $m \geq 3$, then we can consider the element
 \[ [[x_1, a_2],[x_1, a_3]] = \frac{1}{4}[b_{1,2},b_{1,3}] \neq 0,\]
which is clearly a non-trivial element of $\gamma_4(\chi(\mf{h}))$.
\end{proof}

In order to compute the formula for the dimension of $R(\mf{n}_{m,2})$ it suffices to subtract from the dimension of 
$\chi(\mf{n}_{m,2})$, the dimensions of $Im(\rho)$ and $H_2(\mf{n}_{m,2};K)$. The former can be computed by observing that 
\[ Im(\rho) = \{(x,y,z) \in (\mf{n}_{m,2})^3  | \ \ x-y+z \in \mf{n}'_{m,2}\},\]
so $dim(Im(\rho)) = 2 dim(\mf{n}_{m,2}) + dim(\mf{n}'_{m,2})$, and these two quantities are well-known. Regarding the other term, we have:
\[H_2(\mf{n}_{m,2};K)\simeq \gamma_3/\gamma_4\]
where the $\gamma_i$ are the terms of the lower central series of $F$ (the free Lie algebra on $m$ generators). The dimension of $\gamma_3/\gamma_4$ is $\frac{1}{3}(m^3-m)$ by Witt's dimension formula. Putting all of this together, we obtain the polynomial formula for $dim(R)$, as stated in the introduction.

\subsection{Free nilpotent of class 3}
\begin{crlr} 
If $\mf{h} =  \mf{n}_{m,3}$, then $L(\mf{h})$ is a central extension of $K^{m \binom{m}{2}}$, by 
$\mf{n}_{m,4} \oplus \mf{n}_{\binom{m}{2}, 2}$.
\end{crlr}

\begin{proof}
 Once again we must have:
 \begin{equation} \label{c3.1}
 \mu([x_{i_1}, \ldots, x_{i_4}]) = [a_{i_1},a_{i_2}, b_{i_3, i_4}] = 0
 \end{equation}
 and
 \begin{equation} \label{c3.2}
 \mu([x_{i_1}, \ldots, x_{i_5}]) = [a_{i_1}, \ldots, a_{i_5}]=0. 
 \end{equation}
 Also, by the defining relation \eqref{Lrel2.1} we have
 \[ [b_{i_1,i_2}, a_{i_3}, b_{i_4, i_5}] = [[a_{i_1},a_{i_2}], a_{i_3}, a_{i_4},a_{i_5}]\]
Thus, imposing \eqref{c3.1} and \eqref{c3.2} as relators, we get that $[a_i, b_{l,k}]$ is central and that the Lie algebra generated by the $a_i$'s is 
nilpotent of class $4$. Furthermore, again by \eqref{Lrel2.1} we have
\[ [b_{i_1,i_2}, b_{i_3,i_4}, b_{i_5,i_6}] = [[a_{i_1},a_{i_2}], [a_{i_3},a_{i_4}],b_{i_5,i_6}]=0,\]
so the Lie algebra generated by the $b_{i,j}$'s is nilpotent of class $3$. It is clear that the relations of $\mc{L}$ become trivial in the presence
of the relations described in the assumption of the proposition, and also clearly no relations of smaller degree involving the $a_i$ can exist. 
 \end{proof}

 \begin{rem}
  It is immediate in this case that $\chi(\mf{h})$ is nilpotent of class $4$, since $L(\mf{h})$ contains a copy of $\mf{n}_{m,4}$.
 \end{rem}

 \begin{crlr}
 $dim R(\mf{n}_{m,2}) = dim R(\mf{n}_{m,3})$ for all $m$. 
\end{crlr}

\begin{proof}
We can proceed as in the previous subsection and compute the exact dimension of $R(\mf{n}_{m,3})$ in terms of $m$. Here we use Witt's dimension formula to compute both $dim(\mf{n}_{m,c})$, for $c=2,3,4$, and $dimH_2(\mf{n}_{m,3}; K)$.
\end{proof}

  For $c \geq 4$ the situation is more complicated and we cannot expect to describe $\chi(\mf{h})$ as nicely as in the cases above. The reason for this is that the expression of type 
  $\mu(u)=0$ for $u$ of length at least $5$ will not trivialize the defining \eqref{Lrel1}, that is, the elements $[a_i,b_{j,k}]$ will not be central in general. 
\end{section}

%
%
%
%
%

\begin{section}{The ideal \texorpdfstring{$R$}{R}} \label{GSbasis}
Let $\mf{g}$ be free with generators $x_1, x_2, x_3$. Consider the presentation $L(\mf{g}) = \langle X | S\rangle$ described in Section \ref{presL}, where $X = \{a_i, b_{i,j}\}_{i,j}$.  We will assume that relators which are already trivial (in the free Lie algebra with free basis $X$) are not elements of $S$. For instance, for a relator of type \eqref{Lrel1}:
\[s = [a_i, a_j, b_{k,l}] - [a_i, b_{j,k}, a_l] - [a_i, a_k, b_{j,l}],\]
we assume that $j$, $k$ and $l$ are distinct indices.

For each even integer $n$, consider the element
\begin{equation}  \label{deffn}
  f_n = [b_{1,2}, a_2, \ldots, a_2, a_3] - [[a_1,a_2], a_2, \ldots, a_2, b_{2,3}], 
 \end{equation}
where $a_2$ appears $n$ times in each bracket. By applying the homomorphism $\rho$ we deduce that each $f_n$ lies in $R(\mf{g})$ (recall that $R(\mf{g}) = W(\mf{g}) = ker(\rho)$, since $\mf{g}$ is free).

\begin{prop}  \label{propRinfdim}
 The element $f_n$ is non-trivial in $L(\mf{g})$. 
\end{prop}

\begin{proof}
Define $d: X \to \mb{N}$ by $d(a_i)=1$ and $d(b_{i,j})=2$. Then clearly the set $S$ is homogeneous with respect to the degree function extending $d$.  

Let $S_m$ be the set of elements in $S$ with degree at most $m$. Clearly it suffices to show that $f_n$ does not lie in the ideal generated by $S_m$ with $m = d(f_n) = n+3$, because it cannot be a consequence of relators of higher degree than itself.

Consider the weighted deg-lex ordering on the associative words with letters in $X$, where 
\[ b_{1,2} > b_{1,3} > b_{2,3} > a_1> a_2 > a_3\]
and the degrees are defined by the function $d$. 

The set $S_m$ is finite and homogeneous, so we are in the situation described in the end of Section \ref{Shi}. Let $\widehat{S_m}$ be the resulting reduced set with the property that any composition between two of its elements either lies in $\widehat{S_m}$ or has degree greater than $f_n$. This set inherits the following property from $S_m$: if 
$s \in \widehat{S_m}$, then no monomial involved in the expression of $s$ can have only a single ocurrence of $b_{1,2}$ and some occurrences of $a_2$ as letters. Indeed, if $g,h \in \widehat{S_m}$, then any monomial involved in a composition of $g$ and $h$ contains all the letters of some monomial involved in $g$ or $h$, and similarly with reduction. The fact that $S_m$ actually has such property to begin with can be verified directly by inspection of
\eqref{Lrel1}, \eqref{Lrel2.1} and \eqref{Lrel2.2}.

The same reasoning implies that no monomial involved in the expression of any $s \in \widehat{S_m}$ can have only $a_2$ and a single occurrence of $a_3$ as letters (that is, if all letters of the monomial are $a_2$ and $a_3$, then $a_3$ must appear at least twice).

Notice that $\bar{f}_n = b_{1,2} a_2^n a_3$, since the term $[b_{1,2},a_2, \ldots, a_2, a_3]$ is the unique regular bracketing of the regular associative word $b_{1,2} a_2^n a_3$ and the other term of $f_n$ does not involve $b_{1,2}$, which is the highest letter in the lexicographic ordering. Suppose that, for some $s \in \widehat{S_m}$, the associative word $\bar{s}$ is a subword $\bar{f_n} = b_{1,2} a_2^n a_3$. Clearly $\bar{s}$ cannot be of type $a_2^k$. It cannot be neither $\bar{s}=b_{1,2} a_2^k$ nor $\bar{s} = a_2^k a_3$ as well, by the previous paragraph. 

The last thing we must check is that there is no element $s \in \widehat{S_m}$ with $\bar{s}=\bar{f}_n = b_{1,2} a_2^n a_3$. If such an element existed, it would not be an element of $S_m$, because all elements $s_0 \in S_m$ result in a word $\bar{s}_0$ of odd length. Thus $s$ should be the result of some composition or some reduction. In any case in we conclude that there is some $g \in \widehat{S_m}$ of lower degree and a monomial $u$ that has non-zero coefficient in the expression of $g$ and such that all letters of $u$ are letters of $\bar{f}_n$, with at most the same number of occurrences.

First notice that $u$ must involve $a_3$, otherwise we would a have a monomial of some element of $\widehat{S_m}$ involving only $b_{1,2}$ and $a_2$, which we already argued that cannot happen. Similarly, it must involve $b_{1,2}$, otherwise $u$ would a be a monomial with letters $a_2$ and a single occurrence of $a_3$.

So $u$ involves \textit{all} letters of $\bar{f}_n$. The only possibilities for $\bar{u}$ to be regular are $\bar{u} = b_{1,2}a_2^i a_3 a_2^j$ for some $i,j$ with $i+j < n$.
The bracketing of such a word is of the form
\[ u = [ \ldots [[b_{1,2}, a_2, \ldots, a_2, a_3], a_2], \ldots, a_2]\]
where $a_2$ appears $i$ times to the left of $a_3$, and $j$ times to the right. It follows that the monomial corresponding to $\bar{f}_n$ in $s$ is obtained in the composition or reduction process by taking brackets of $u$ with $x_2$ on the right or on the left. In any case the resulting monomial is up to sign the regular bracketing of the word $b_{1,2}a_2^i a_3 a_2^k$ for some $k>j$, so $u=b_{1,2} a_2^n a_3$ can never be achieved. Thus $\bar{f}_n$ cannot actually be the associative carrier of a monomial involved in some composition or reduction of elements of $\widehat{S_m}$.

By Theorem \ref{thmShirshov} it follows that $f_n$ does not lie in the ideal generated by $S_m$, and consequently $f_n \notin \langle \langle S \rangle \rangle$. Thus $f_n$, as an element of $L(\mf{g})$, is non-trivial.
\end{proof}

\begin{theorem}  \label{6.2}
 If $\mf{f}$ is free non-abelian of rank at least $3$, then $R(\mf{g})$ is infinite-dimensional.
\end{theorem}

\begin{proof}
 For a free Lie algebra $\mf{g}$ of rank $3$, Proposition \ref{propRinfdim} says that none of the $f_n$ defined in \eqref{deffn} are trivial. Furthermore, they are all of different degree with respect to the function $d$ defined in the proof of the proposition, so they make up an infinite linearly independent set inside $R(\mf{g})$. In general, if $\mf{f}$ is free of rank more than $3$, then there is an epimorphism  $\phi: \mf{f} \to \mf{g}$ and the induced homomorphism $\phi_{\ast}: \chi(\mf{f}) \to \chi(\mf{g})$ satisfies $\phi_{\ast}(R(\mf{f})) = R(\mf{g})$, thus $R(\mf{f})$ is of infinite dimension as well.
\end{proof}

\begin{crlr}
If $\mf{g}$ is free non-abelian of rank at least $3$, then $\chi(\mf{g})$ is of infinite cohomological dimension.
\end{crlr}

\begin{proof}
 This is clear, since $\chi(\mf{g})$ contains an abelian subalgebra of infinite dimension. 
\end{proof}

\begin{rem}
 It is clear by the proofs that Theorem \ref{6.2} and its corollary hold if we assume only that the free Lie algebra of rank $3$ is a quotient of $\mf{g}$.
\end{rem}

\end{section}

\section*{Acknowledgements}
The author is supported by grant 2015/22064-6 from S\~{a}o Paulo Research Foundation (FAPESP).

\end{document}